\documentclass[10pt]{amsart}
\usepackage{amssymb, amscd, amsmath, amsthm, epsfig, latexsym, enumerate}

\renewcommand{\geq}{\geqslant}
\renewcommand{\leq}{\leqslant}

\newtheorem{theorem}{Theorem}
\newtheorem{lemma}[theorem]{Lemma}

\newtheorem*{thm}{Theorem}

\newtheorem*{cor*}{Corollary}

\begin{document}
\title{the  $\mathbb{F}_2$-cohomology rings of 3-manifolds}

\author{Jonathan A. Hillman}
\address{School of Mathematics and Statistics\\
     University of Sydney, NSW 2006\\
      Australia }

\email{jonathanhillman47@gmail.com}

\begin{abstract}
We give a new argument for the characterization of the cohomology rings of closed 3-manifolds with coefficients $\mathbb{F}_2$, 
first given by M. M. Postnikov (in terms of intersection rings) in 1948.
\end{abstract}

\keywords{3-manifold, cohomology ring,  {\it mod}-(2) coefficients}

\subjclass{57M45}


\maketitle
M. M. Postnikov characterized the $\mathbb{F}_2$-homology intersection
rings of closed 3-manifolds in his first published paper \cite{Po48},
and D. Sullivan determined the $\mathbb{Z}$-cohomology rings of closed orientable 
3-manifolds in \cite{Su75}.
The orientable case was settled comprehensively by V. G. Turaev,
who considered not only cohomology with coefficients $\mathbb{Z}/n\mathbb{Z}$ 
for all $n\geq0$,
but determined the interactions of these cohomology rings
with each other and with the torsion linking pairing \cite{Tu83}.

Postnikov showed that the $\mathbb{F}_2$-homology intersection ring of a closed 
3-manifold $M$
is a finite graded $\mathbb{F}_2$-algebra which satisfies 3-dimensional Poincar\'e duality
and the``Postnikov-Wu identity" with respect to a distinguished element $w=w_1(M)$,
and conversely every such ``$MS$-algebra" is isomorphic to such a cohomology ring.
(The terms in quotation marks are defined in \S5 below.)
The constructive part of his argument used induction 
on the rank of the degree-1 component of the algebra,
with bases corresponding to $S^3$ and $S^2\tilde\times{S^1}$,
for the cases $w=0$ and $w\not=0$, respectively.
The inductive step in \cite{Po48} used (in today's terminology) surgery 
on a knot $K$ in a 3-manifold $M$ such that $K$ has trivial image 
in $H_1(M;\mathbb{F}_2)$. 

We review briefly the results of Sullivan and Turaev in \S2 and \S3.
The main body of the paper is in \S4--\S9,
where we give a new argument for Postnikov's result.
The basic idea (superposition of elementary models) is quite simple, 
but its application to nonorientable 3-manifolds involves some effort.
Although we shall use the language of cohomology,  our  calculations rest 
largely upon the intersections of curves and surfaces in a 3-manifold, as in Postnikov's original account.
Our contributions are merely to give a direct link presentation of a suitable 3-manifold, 
suppressing the induction,
and to make more explicit the penultimate paragraph of Postnikov's argument,
which consists of three sentences,
beginning with  ``{\it Okazivayetsya, shto} $\dots$", 
meaning roughly ``It turns out that $\dots$".

When the homology is torsion free the $\mathbb{F}_p$-cohomology ring 
is determined by the integral cohomology ring.
In the final two sections we show that this is not always so when there is torsion.

\section{notation}

If $A$ is a finitely generated abelian group let $tA$ be its torsion subgroup,
and let $A^*=Hom(A,\mathbb{Z})$ and $A^{*n}=Hom(A,\mathbb{Z}/n\mathbb{Z})$,  
for each $n>1$.
If $R$ is another abelian group then there is a canonical epimorphism from $Hom(A,R)$ 
onto $Hom(tA,R)$,  with kernel $Hom(A/tA,R)$.

If $R$ is a commutative ring and $\{e_1,\dots,e_d\}$ is a basis for a free $R$-module $V$
then the Kronecker dual basis for $V^*=Hom_R(V,R)$ is the basis $\{e_1^*,\dots,e_d^*\}$
determined by $e_i^*(e_i)=1$ and $e_i^*(e_j)=0$ for all $j\not=i$ and all $i$.

If $G$ is a group let $G'$ and $\zeta{G}$ be the commutator subgroup and
centre of $G$, and let $G^{ab}=G/G'$ be its  abelianization.
Let $I(G)$ be the preimage in $G$ of $tG^{ab}$.
Let $X^r(G)$ be the verbal subgroup generated by all $r$th powers $g^r$ with $g\in{G}$.

If $L$ is an $m$-component link let $L_i$ be its $i$th component, 
and let $M(L;\mathcal{F})$ be the closed 3-manifold
obtained by $\mathcal{F}$-framed surgery on $L$.
If each component has the 0-framing we shall write just $M(L)$,
and if each component has framing with slope $p$ we shall write $M(L;p)$.
Spanning surfaces for link components and closed surfaces representing 
Poincar\'e duals of classes in $H^1(M;\mathbb{F}_2)$ may be nonorientable,
and we shall not comment further on this possibility.

Let $L_{2,4}$ and $Bo$ be the (2,4)-torus link and the Borromean rings 3-component link.
(These are the links $4^2_1$ and $6^3_2$ in Rolfsen's tables \cite{Rolf}.)
Let $Bo(n)$ be the link obtained by replacing $Bo_3$
of $Bo$ by its $(1,n)$-cable.

If $A\in{GL(2,\mathbb{Z})}$ then $MT(A)$ is the mapping torus
of the self-homeomorphism of the torus $T$ induced by $A$.

\section{characteristic not 2}

Let $M$ be a closed 3-manifold and let $\pi=\pi_1(M)$.
Let $R=\mathbb{Z}$,  a prime field $\mathbb{Q}$ or $\mathbb{F}_p$, with $p$ odd.
Let $H=\pi^{ab}$ and let $H^*=Hom(H,R)$.
If $M$ is orientable then $H_3(M;R)\cong{R}$, 
and cap product with a fundamental class $[M]$ 
defines Poincar\'e duality isomorphisms $D_2:H^2(M;R)\to{R\otimes{H}}$
and $D_3:H^3(M;R)\to{R}$,
while cup product and duality define homomorphisms
$\gamma:\wedge_2H^*\to{H^2(M;R)}$ and 
$\mu=\mu_M:\wedge_3H^*\to{R}$.
(Alternating trilinear functions such as $\mu$ may be identified with 
elements of $\wedge_3H$.)
These satisfy the equations
\[
D_3(a\smile{D_2^{-1}(h)})=a(h)\quad\forall~a\in{H^*}~and~h\in{H}
\]
and
\[
\mu(a\wedge{b}\wedge{c})=a(D_2\gamma(b\wedge{c}))\quad
\forall~a,b,c\in{H^*}.
\]
If $R$ is a prime field, (with characteristic $\not=2$) 
or if $R=\mathbb{Z}$ and $H$ is torsion-free then the cohomology ring $H^*(M;R)$ 
is determined by $H$, $\mu$ and the duality isomorphisms $D_2$ and $D_3$,
via these equations.
(More explicitly, let $\mathcal{H}^0=R$,
$\mathcal{H}^1=H^*$, $\mathcal{H}^2=H$ and $\mathcal{H}^3=R\varepsilon_3$,
and let $rs$ be the unique solution of 
``$t(rs)=\mu(r\wedge{s}\wedge{t})\varepsilon_3~\forall~t\in{R^1}$",
and $rh=r(h)\varepsilon_3$,
for $r,s\in\mathcal{H}^1$ and $h\in\mathcal{H}^2$.
Then $rst=\mu(r\wedge{s}\wedge{t})\varepsilon_3$,  for all $r,s,t\in{H^*}$,
and $\mathcal{H}^*(H,\mu)=\oplus_{i=0}^3\mathcal{H}^i$ is a graded ring.
We may use $D_2$ and $D_3$ to determine an isomorphism 
$H^*(M;R)\cong\mathcal{H}^*(H,\mu)$.)
If $H^*\not=0$ then $D_2$ determines $D_3$, by the first equation above.

\begin{thm} [Sullivan \cite{Su75}]
If $R=\mathbb{Z}$ then every such pair $(H,\mu)$ is realizable by some closed 
orientable $3$-manifold with torsion-free homology.
\qed
\end{thm}

We may assume that the 3-manifold is irreducible \cite{Li81}.

If $M$ is  not orientable and $F$ is a field of characteristic $\not=2$ 
hen $H^3(M;F)=0$, and so $\beta_2(M;F)=\beta_1(M;F)-1$.
Up to automorphisms of $H^1(M;F)$ and $H^2(M;F)$, 
the only issue of interest is the rank of $c_p$.
If $P_c=\#^r\mathbb{RP}^2$ and $M=S^1\times{P_c}$ then $c_p$
is an epimorphism for all characteristics $p\not=2$.
Hence connected sums of $S^1\times{P_c}$
with copies of $S^1\times{S^2}$ realize all possible combinations of rank 
and Betti numbers satisfying $rk(c_p)\leq\beta_2(M;\mathbb{F})-1$.

When $R=\mathbb{Z}$ or $\mathbb{F}_2$ the Poincar\'e duals of Kronecker duals of a basis for $H_1(M;R)$ represented by simple closed curves in $M$ may be represented 
by closed surfaces in $M$ which meet one such curve transversely in one point and
are disjoint from the other curves. However these closed surfaces are generally not pairwise disjoint, and we may use their intersections to identify cup products.
See \cite{Hu11}.

\section{integer coefficients: torsion}

When $R=\mathbb{Z}$ and $H_1(M;\mathbb{Z})$ has nontrivial torsion 
then $\gamma$ is no longer determined by $\mu_M$ and the above equation.
The torsion subgroup has a complementary direct summand,
but the splitting is not canonical.
Cup products with coefficients in other rings and their compatibilities 
with integral cup product must also be considered.
V.G.Turaev gave a definitive account in \cite{Tu83}.

For each $n>1$ let $\nu_n=\nu_{nM}:\wedge_3H^1(M;\mathbb{Z}/n\mathbb{Z})
\to\mathbb{Z}/n\mathbb{Z}$ be defined by
\[
\nu_n(X,Y,Z)=(X\cup{Y}\cup{Z})\cap[M], \quad\mathrm{for~all}~X,Y,Z\in{H^1(M;\mathbb{Z}/n\mathbb{Z})}.
\]
Then $\nu_{nM}$ and Poincar\'e duality together determine the ring
$H^*(M;\mathbb{Z}/n\mathbb{Z})$.
Every 3-form $\nu:\wedge_3(\mathbb{Z}/n\mathbb{Z})^\beta
\to\mathbb{Z}/n\mathbb{Z}$
lifts to a 3-form $\widehat{\nu}:\wedge_3\mathbb{Z}^\beta\to\mathbb{Z}$.
Hence it is an immediate consequence of Sullivan's construction that 
every such 3-form $\nu$ can be realized as $\nu_{nM}$ for some closed 
orientable 3-manifold $M$ with $H_1(M;\mathbb{Z})\cong\mathbb{Z}^\beta$.
If $p$ is an odd prime it follows that every finite graded-commutative graded 
$\mathbb{F}_p$-algebra satisfying
3-dimensional Poincar\'e duality is the $\mathbb{F}_p$ cohomology ring
of such a 3-manifold.

The Bockstein homomorphism
$\beta_{\mathbb{Q}/\mathbb{Z}}:Hom(H,\mathbb{Q}/\mathbb{Z})\to{H^2(M;\mathbb{Z})}$
has image $Ext(H,\mathbb{Z})$,  the torsion subgroup of $H^2(M;\mathbb{Z})$. 
We may use this to define the torsion linking  pairing 
$\ell:tH\times{tH}\to\mathbb{Q}/\mathbb{Z}$ 
by 
\[
\ell(u,v)=(D_2\circ\beta_{\mathbb{Q}/\mathbb{Z}})^{-1}(v)(u)\quad\forall~u,v\in{tH}.
\]
This pairing is nonsingular and symmetric,
and $\ell$ and $\beta_{\mathbb{Q}/\mathbb{Z}}$ determine each other
(given $D_2$).
Every such pairing is realizable by some  
$\mathbb{Q}$-homology 3-sphere \cite[Theorem 6.1]{KK80}.
Taking connected sums shows that every such triple $(H, \mu, \ell)$ 
is realizable by some 3-manifold. 
On the other hand,  $\mu$ and $\ell$ are independent invariants.

It is easily verified that if $x\in{H^1(M;\mathbb{Z}/n\mathbb{Z})}$
then $x^2=\frac{n}2\beta_{\mathbb{Q}/\mathbb{Z}}(x)$.
(It suffices to check this for $x=id_{\mathbb{Z}/n\mathbb{Z}}$,
considered as an element of $H^1(\mathbb{Z}/n\mathbb{Z};\mathbb{Q}/\mathbb{Z})=
H^1(\mathbb{Z}/n\mathbb{Z};\mathbb{Z}/n\mathbb{Z})$.)
Let $\psi_n:\mathbb{Z}/n\mathbb{Z}\to\mathbb{Q}/\mathbb{Z}$ be the standard inclusion.
If $x\in{H^{*n}}$ let $\hat{x}$ be the element of $tH$ such that 
$\ell(\hat{x},a)=\psi_n(x(a))$ for all $a\in{tH}$.
Turaev showed that
\[
\psi_n(\nu_n(x,x,y))=\frac{n}2\ell(\hat{x},\hat{y})~\mathrm{for~all}
~x,y\in{H^{*n}}.
\]
When $n=2$ this condition implies the orientable case of the 
Postnikov-Wu identity invoked below.
(Note that if $n$ is odd then both sides are 0.)

Let $A$ be a finitely generated abelian group, 
let $\ell:tA\times{tA}\to\mathbb{Q}/\mathbb{Z}$ be a nonsingular symmetric pairing,
and suppose that $\nu:\wedge_3A^*\to\mathbb{Z}$
and $\nu_n:\wedge_3A^{*n}\to\mathbb{Z}/n\mathbb{Z}$
is a system of alternating trilinear functions
which are compatible under reduction {\it mod}-$n$.
Then Turaev showed that such a group $A$, pairing $\ell$ and system 
of trilinear functions may be realizable by the homology and cohomology of a closed orientable 3-manifold if and only if the above condition deriving from the interaction of 
the Bockstein $\beta_{\mathbb{Q}/\mathbb{Z}}$ with the cup-square holds 
for all $n$ dividing the order of $tA$ \cite{Tu83}.

\section{surgery on links}

In the next six sections we shall give an alternate approach to Postnikov's result
on the $\mathbb{F}_2$-cohomology rings of 3-manifolds.
Our examples shall be constructed by surgery on framed links.
Related constructions were used by Turaev and Postnikov, 
although in the latter case ``surgery" was not yet a mathematical term.
(Sullivan uses instead Heegaard decompositions.)

In the orientable case, 
in order to construct a closed 3-manifold with $M$ with $\beta_1(M;\mathbb{F}_2)=\rho$ we modify various of the components of a trivial $\rho$-component link $\rho{U}$ 
in $S^3$, using local moves which involve replacing trivial 2- or 3-component
tangles in a ball by other tangles from a limited repertoire.

The components of the tangles in Figure 1 represent distinct link components.
Move $(a)$ changes the linking number by $\pm2$, and
an application of this move changes the 2-component trivial link $2U$ 
into $L_{2,4}$.
Move $(b)$ does not changing the linking number.
An application of this move changes the 3-component trivial link $3U$ into $Bo$.
Moves $(a)$ and $(b)$ (in either order) change $3U$ into the link of Figure 3.
These moves do not change the knot types of the link components.

\setlength{\unitlength}{1mm}
\begin{picture}(90,58)(5,-5)
{
\linethickness{1pt}
\put(10,12){\line(0,1){34}}
}

\put(16,12){\line(0,1){34}}

\put(25,25){$\to$}

{
\linethickness{1pt}
\put(40,12){\line(0,1){8}}
\put(40,22){\line(0,1){10}}
\put(40,34){\line(0,1){12}}
}

\put(46,12){\line(0,1){6}}
\qbezier(46,18)(46,21)(43,21)
\put(39,21){\line(1,0){4}}
\qbezier(39,21)(36,21)(36,24)
\qbezier(36,24)(36,27)(39,27)
\qbezier(41,27)(44,27)(44,30)
\qbezier(44,30)(44,33)(41,33)
\qbezier(39,33)(36,33)(36,36)
\put(39,33){\line(1,0){2}}
\qbezier(36,36)(36,39)(39,39)
\put(41,39){\line(1,0){2}}
\qbezier(43,39)(46,39)(46,42)
\put(46,42){\line(0,1){4}}

\put(25,5){$(a)$}

{\linethickness{1pt}
\qbezier(66,40)(71,30)(66,20)
\qbezier(84,40)(79,30)(84,20)}

\qbezier(70,40)(75,35)(80,40)

\put(92,30){$\to$}

{\linethickness{1pt}
\qbezier(102,40)(105,37)(105,36)
\qbezier(102,20)(105,23)(106,28)
\qbezier(106,28)(106.5,30)(106,32)

\qbezier(123,40)(120,37)(119.5,35.2)
\qbezier(119,25.5)(118.5,30)(119,33.2)
\qbezier(119.5,23.5)(120,21.5)(121,20)
}

\qbezier(100,30)(100,33)(105,35)
\qbezier(100,30)(100,25)(103,24)
\qbezier(102,30)(102,32)(105,33)
\qbezier(102,30)(102,27)(104,26)

\qbezier(105,33)(111,35)(118,32.7)
\qbezier(105,35)(109,36)(110,40)
\qbezier(119,34.5)(115,35.5)(114,40)

\qbezier(106,25.5)(111,24)(118,26)
\qbezier(105,23.4)(112,22)(119,24.3)

\qbezier(119,34.5)(120,34)(121,33.5)
\qbezier(119.7,32)(120,32)(120.2, 31.8)
\qbezier(121,33.5)(122,32)(120,31.8)

\qbezier(119,24.3)(120,24.5)(121,25.5)
\qbezier(119.8,26.5)(120.1,26.8)(120.5,26.7)
\qbezier(120.5,26.7)(121.5,26.1 )(121,25.5)

\put(92,5){$(b)$}

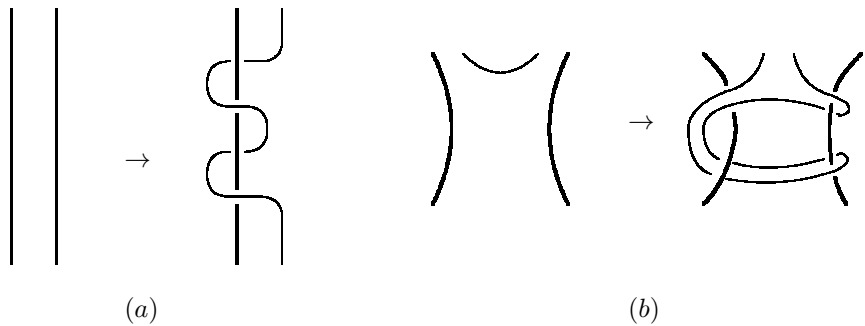
\put(40,0){ Figure 1. Two tangle moves.}

\end{picture}

In each of these cases a given set of spanning surfaces for the components 
of the original link may be modified within the ball containing the tangle 
without introducing new intersections outside the ball.
(For instance, applying move $(b)$ to the trivial 3-component link $3U$ gives $Bo$, 
as in Figure 2.)

In the nonorientable case we start with nontrivial links in $S^2\tilde\times{S^1}$
which are disjoint unions of several basic 2- or 3-component links 
with all components orientation preserving.
(Each component $K$ then has ``longitudes" which bound surfaces in 
$S^2\tilde\times{S^1}$. )
We may assume that each component meets a fixed copy of the fibre $S^2$
transversely, and shall make our modifications to trivial 2- or 3-component tangles in 
balls lying in the complementary region $S^2\times(0,1)$.

\begin{lemma}
Let $M=M(Bo(n))$, and let $a,b,c$ be the basis for
$H_1(M;\mathbb{Z})$ determined by the meridians of $Bo(n)$.
Let $a^*,b^*,c^*$ be the Kronecker dual basis for $H^1(M;\mathbb{Z})$. 
Then $(a^*\cup{b^*}\cup{c^*})\cap[M]=n$.
\end{lemma}

\begin{proof}
This is most easily seen using a mixture of algebra and geometry.
The first two components bound disjoint discs in their mutual complement.
Each of these discs meets the third component in 2 points.
If we delete neighbourhoods of these intersection points and attach tubes 
which surround arcs of the third component, 
we obtain punctured tori in the exterior of $Bo(n)$,  
which may closed off in $M$ by copies of the surgery discs used in forming $M$.
The resulting tori represent the Poincar\'e duals of $a^*$ and $b^*$,
and intersect along a meridian for the third component of $Bo$,
as in Figure 2.
The latter meridian is homologous in $M$ to $n.c$,
since the third component of $Bo(n)$ is the $(1,n)$-cable of the third component of $Bo$.
Hence $(a^*\cup{b^*}\cup{c^*})\cap[M]=n$.
\end{proof}

If $H$ has basis $\{e_1,\dots,e_\beta\}$ 
then the simple 3-forms $e_{ijk}=e_i\wedge{e_j}\wedge{e_k}$
with $i<j<k$ form a basis for $\wedge_3H$.

\begin{theorem}
Let $M=M(L)$, where $L$ is an $m$-component link with ordered 
and oriented components and such that each $3$-component sublink 
is either trivial or is a copy of $Bo$.
Let $I$ be the set of ordered triples $i<j<k$ such that $L_i\cup{L_j}\cup{L_k}$
is isotopic to $Bo$, and let $\{a_i\}$ be the meridianal basis of $H_1(M;\mathbb{Z})$. 
Then 
\[
\mu_M=\Sigma_{(i,j,k)\in{I}}a_i\wedge{a_j}\wedge{a_k}.
\]
\end{theorem}

\begin{proof}
Let $e_1,\dots,e_m$ be the basis for $H_1(M;\mathbb{Z})$ determined by the meridians for $L$, and let $e_1^*,\dots,e_m^*$ be the Kronecker dual
basis for $H^1(M;\mathbb{Z})$.
To calculate $\mu(e_i^*,e_j^*,e_k^*)$ it suffices to use Poincar\'e duals for
the $e_i^*$ as in Lemma 1.
\end{proof}

Figure 2 shows two punctured tori spanning components $Bo_1$ and $Bo_2$ of $Bo$, 
with intersection a meridian $c$ for the third component.

\setlength{\unitlength}{1mm}
\begin{picture}(95,85)(-30,-18)

\put(7,54){$L_1$}
\put(-13,20){$L_2$}
\put(33,35.5){$c$}

\linethickness{1pt}
\qbezier(10,40)(10,60)(30,60)
\put(10,28){\line(0,1){12}}
\qbezier(10,28)(10,18),(20,18) 

\put(32,18){\line(1,0){24}}

\put(50,6){\line(1,0){5}}
\put(50,9){\line(1,0){3}}
\qbezier(55,6)(61,6)(61,12)
\qbezier(53,9)(58,9)(58,14)
\put(58,14){\line(0,1){10}}
\put(61,12){\line(0,1){12}}
\qbezier(58,24)(58,25.5)(59.5,25.5)
\qbezier(59.5,25.5)(61,25.5)(61,24)

\put(63,18){\line(1,0){2}}
\qbezier(65,18)(75,18)(75,28)
\put(75,28){\line(0,1){12}}
\qbezier(75,40)(75,60)(55,60)
\put(30,60){\line(1,0){25}}

\qbezier(34,34)(37.5,34)(37.5, 30.5)
\qbezier(34,27)(37.5,27)(37.5,30.5)

\qbezier(34,34)(32,34)(31.8,32.8)
\qbezier(31,31.5)(30.6,30.4)(31,29.3)
\qbezier(34,27)(32,27)(31.7,28)

\qbezier(22,22)(22,34)(34,34)
\put(22,17){\line(0,1){5}}
\qbezier(22,17)(22,4)(35,4)

\qbezier(30,22)(30,27)(35,27)
\put(30,16){\line(0,1){6}}
\qbezier(30,16)(30,11)(35,11)

\qbezier(35,11)(38.5,11)(38.5, 7.5)
\qbezier(35,4)(38.5,4)(38.5,7.5)

\qbezier(-7,30)(-7,45)(8,45)

\qbezier(-7,10)(-7,-5)(8,-5)
\put(-7,10){\line(0,1){20}}
\put(8,-5){\line(1,0){25}}

\put(48, 10){\line(0,1){6}}
\qbezier(48,10)(48,-5)(33,-5)

\thinlines

\put(12,45){\line(1,0){21}}

\qbezier(42,34)(45.5,34)(45.5, 30.5)
\qbezier(42,27)(45.5,27)(45.5,30.5)
\put(37,34){\line(1,0){1.5}}
\put(37,27){\line(1,0){1.5}}
\put(39.5,34){\line(1,0){1.5}}
\put(39.5,27){\line(1,0){1.5}}

\qbezier(33,45)(48,45)(48,30)
\put(48,20){\line(0,1){10}}

\qbezier(25,14)(25,6)(33,6)
\qbezier(28,14)(28,9)(33,9)
\put(33,6){\line(1,0){4}}
\put(33,9){\line(1,0){4}}

\put(25,14){\line(0,1){9}}
\put(28,14){\line(0,1){9}}
\qbezier(25,23)(25,32)(34,32)
\qbezier(28,23)(28,29)(34,29)
\qbezier(34,32)(35.5,32)(35.5,30.5)
\qbezier(34,29)(35.5,29)(35.5,30.5)

\put(39.5,6){\line(1,0){1.5}}
\put(42,6){\line(1,0){1.5}}
\put(44.5,6){\line(1,0){1.5}}
\put(39.5,9){\line(1,0){1.5}}
\put(42,9){\line(1,0){1.5}}
\put(44.5,9){\line(1,0){1.5}}

\put(25,-13){Figure 2.}
\end{picture}

Lemma 1 and Theorem 2 suffice to recover Sullivan's result 
for $H\cong\mathbb{Z}^\beta$ with $\beta\leq5$.
If $\beta<3$ then $\mu=0$ and  if $\beta=3$ then 
$\mu=re_1^*\wedge{e_2^*}\wedge{e_3^*}$, for some $r\in\mathbb{Z}$.
If $\beta=4$ then 3-forms are dual to 1-forms, 
while if $\beta=5$ then 3-forms are dual to alternating 2-forms,
and so $\mu$ is equivalent (under the action of
$GL(\beta,\mathbb{Z})$) to $re_1^*\wedge{e_2^*}\wedge{e_3^*}$ or
$re_1^*\wedge{e_2^*}\wedge{e_3^*}+se_1^*\wedge{e_4^*}\wedge{e_5^*}$
(respectively),
for some $r,s\in\mathbb{Z}$.

If $H^{*p}$ has basis $\{e_1,\dots,e_\gamma\}$ 
then the simple 3-forms $e_{ijk}=e_i\wedge{e_j}\wedge{e_k}$
with $i<j<k$ form a basis for $\wedge_3H^{*p}$.
If $\nu=\Sigma\eta_{ijk}{e_{ijk}}$, where each  $\eta_{ijk}=0$ or 1
then we may realize $\nu$ by a closed orientable 3-manifold with
$H_1(M;\mathbb{Z})\cong\mathbb{F}_p^\gamma$.

A similar argument applies for $M(Bo(n);2)$ and coefficients $\mathbb{F}_2$
(and probably also to $M(Bo(n);p)$ and coefficients $\mathbb{F}_p$),
but we may need to extend intersection theory to 
$\mathbb{Z}/p\mathbb{Z}$-manifolds  \cite{Su67}, 
or alternatively use a more algebraic argument, working modulo
the cores $S^1\times\{0\}$ of the surgeries.

This construction could then be extended to the case 
$H\cong\mathbb{Z}^\beta\oplus\mathbb{F}_p^\rho$,
by using 0-framed surgeries on the first $\beta$ components.
However another idea seems necessary to pick up coefficients other than 0 or 1.

If $F$ is a field and $\beta\leq7$ then every 3-form on $F^\beta$ is equivalent 
under the action of  $GL(\beta,F)$ to a standard 3-form of the above type.
This is clear if $\beta\leq3$, and also if $\beta=4$, 
for then 3-forms are dual to 1-forms.
If $\beta=5$ then 3-forms are dual to 2-forms, 
and so correspond to skew-symmetric pairings.
Hence there are just two equivalence classes of nonzero forms,
represented by $e_{123}$ and $e_{123}+e_{145}$.
The case $\beta=6$ involves more work, but there are just two more standard forms,
represented by $e_{123}+e_{456}$ and $e_{162}+e_{243}+e_{135}$.
If $\beta=7$ there are at most 12 equivalence classes, 
some involving coefficients other than 0 or 1,
but the above construction still suffices. 
See \cite{CH88} for details of the standard forms in this case.
On the other hand, if $\beta>8$ then 
$\dim_FGL(\beta,F)=\beta^2<\dim_F\wedge_3F^\beta=\binom\beta3$,
and the number of equivalence classes is unbounded as the order of $F$ increases.
The result of Turaev shows that we should not need to appeal to \cite{CH88}.

\section{characteristic 2}

In the strictly antisymmetric cases (characteristic 0 or odd), 
there is only one possible nonzero cup product involving three given degree-1 classes.
When  $R=\mathbb{F}_2$  the homomorphisms $\gamma$ and $\mu$ must be replaced
by homomorphisms from the symmetric  products 
$\odot^2H^*\to{H^2(M;\mathbb{F}_2)}$ and $\odot^3H^*\to{H^3(M;\mathbb{F}_2)}$.
There are many more possibilities for nonzero triple products,
and we should now consider also nonorientable 3-manifolds.
Postnikov gave a complete account of the $\mathbb{F}_2$-intersection rings 
of 3-manifolds.
He did not assume orientability, 
and his description of $H^*(M;\mathbb{F}_2)$ includes also the relation 
$uvw=u^2v+uv^2$, for all $u,v\in{H^1(M;\mathbb{F}_2)}$ \cite{Po48}.
(Here $w$ is the orientation character.)
This may now be seen as an application  of the Wu relation $Sq^1z=wz$,  
for all $z\in{H^{n-1}(M^n;\mathbb{F}_2)}$, which first appeared later \cite{Wu51}.

We have adapted Postnikov's terminology to the cohomological formulation.
An {\it $MS$-algebra\/} is a finite commutative graded $\mathbb{F}_2$-algebra 
$\mathcal{A}^*=\oplus_{i=0}^3\mathcal{A}^i$ such that 
$\dim\mathcal{A}^0=\dim\mathcal{A}^3=1$,
$\mathcal{A}^j=0$ for $j>3$, 
multiplication defines a perfect pairing from 
$\mathcal{A}^1\times\mathcal{A}^2$ into $\mathcal{A}^3$, 
so that $\mathcal{A}^*$ satisfies formal Poincar\'e duality of dimension 3,
and which has a distinguished element $w\in\mathcal{A}^1$  such that
the {\it Postnikov-Wu identity}
 \[
wxy=x^2y+xy^2
\]
holds for all $x,y\in\mathcal{A}^1$.
An {\it $MS$-algebra isomorphism\/} is a ring isomorphism under which the 
distinguished elements $w$ correspond,
and $\mathcal{A}^*$ is {\it orientable\/} if $w=0$.
We shall abbreviate ``by the nonsingularity of multiplication from 
$\mathcal{A}^1\times\mathcal{A}^2$ to $\mathcal{A}^3$" to ``by nonsingularity".

In the orientable case the Postnikov-Wu identity is an easy consequence 
of standard facts about the reduced Bockstein homomorphisms $\overline\beta_2$.
If $X$ is an orientable $PD_3$-complex then reduction mod $(2)$ maps
$H^2(X;\mathbb{Z})$ onto $H^2(X;\mathbb{F}_2)$, 
since $H^3(X;\mathbb{Z})\cong\mathbb{Z}$.
Hence $\overline\beta_2$ is trivial on  $H^2(X;\mathbb{F}_2)$, 
and so $u^2v+uv^2=Sq^1(u)v+uSq^1(v)=\overline\beta_2(uv)=0$
for all $u,v\in{H^1(X;\mathbb{F}_2)}$.

\section{the orientable case}

In this section we shall consider the orientable case,  which is somewhat easier, 
as we may realize a basis for $\mathcal{A}^1$ 
by meridians for a suitable link in $S^3$.
The Postnikov-Wu identity also simplifies to  $x^2y=xy^2$  
for all $x,y\in\mathcal{A}^1$.
We shall design a framed link representing the 3-manifold, 
guided by the nonzero triple products.
The basic ingredients are $L_{2,4}$, $Bo$
and the link of Figure 3 below.

Let $K$ be a knot in $S^3$ with exterior $X(K)$,
meridian $\mu_K$ and longitude $\lambda_K$.
A $p$-framed surgery on  $K$ is determined by  a homeomorphism 
$\phi:\partial{X(K)}\to{S^1\times{D^2}}$ such that 
$\phi(\lambda_K+p\mu_K)=\partial{D^2}$.
After composition with a self-homeomorphism of $S^1\times{D^2}$,
if necessary, we may assume that $\phi(\mu_K)=S=S^1\times\{*\}$.
It then follows that $\phi(\lambda_K)$ is a simple closed curve $C_p$
representing the homology class $[\partial{D^2}]-p[S]$ in $H_1(S^1\times\partial{D^2})$.
If $p=0$ this curve clearly bounds a copy of $D^2$;
if $p=\pm2$ then it bounds a ribbon with a half-twist, 
i.e., a copy of the M\"obius band $Mb$.
[In general, $C_p$ bounds a $\mathbb{Z}/p\mathbb{Z}$-manifold.
If $p$ is even is there a natural desingularization?]

Let $M=M(L;\mathcal{F})$, where $L$ is an $m$-component link in which all pairwise linking numbers are even and in which each component has even framing.
The meridians for the link components represent a canonical basis for 
$H_1(M;\mathbb{F}_2)$, and the Kronecker duals of this basis give a
basis $\{x_1,\dots,x_m\}$ for $H^1(M;\mathbb{F}_2)$.
Let $F_i$ be a Seifert surface for the the $i$-th component $L_i$ in $S^3$.
The Poincar\'e dual of $x_i$ in $H_2(M;\mathbb{F}_2)$ is represented by the surface 
$\widehat{F}_i$ which is the union of $F_i$ with a 
spanning surface for the longitude $\lambda_{L_i}$ in $S^1\times{D^2}$.
If all the linking numbers are 0 or $\pm2$ then after attaching handles to $F_i$ if necessary,
we may assume that it does not meet any other component of $L$.

If all linking numbers are 0 then we may assume that $F_i$ is orientable,
and then $x_i^2$ and $x_i^3$ are supported by the cocore of the surgery.
In this case $x_i^2=0$ if $L_i$ has framing 0, 
while $x_i^3\not=0$ if the framing is $\pm2$.
Moreover the triple products $x_i^2x_j$ and $x_ix_jx_k$ then depend only on 
the sublinks $L_i\cup{L_j}$ and $L_i\cup{L_j}\cup{L_k}$ involved 
and the framings of these components.
(If the framings and linking numbers are even but some are not 0 or $\pm2$ then 
the components bounds immersed spanning surfaces which are disjoint 
from the other components.
Our approach might extend to these cases, 
but we do not know whether the intersection theory applies adequately 
for transversely immersed surfaces.)

The (2,4)-torus link is the simplest 2-component link with linking number 2.
Its link group has the presentation 
$\langle{a,b}\mid{bab^{-1}a=ab^{-1}ab}\rangle$.
The longitudes are $\ell_a=aba^{-1}b$ and $\ell_b=bab^{-1}a$,
and 0-framed surgery gives the quaternionic 3-manifold $S^3/Q(8)$.

Let $u$ and $v$ be the Kronecker duals of the images of the meridians $a$ and $b$
in $H_1(S^3/Q(8);\mathbb{F}_2)$.
Then $u^2$ and $v^2$ are nonzero, 
since $Q(8)^{ab}$ has exponent 2.
On the other hand, $u^3=v^3=0$,
and so $u^2\not=v^2$ and $u^2v=uv^2\not=0$,  
by nonsingularity.
In this case the kernel of cup product is generated by $u\odot{u}+u\odot{v}+v\odot{v}$.
These are standard facts about $H^*(Q(8);\mathbb{F}_2)$, 
but can be explained in our terms as follows.
The Poincar\'e dual of $u$ is represented by a Klein bottle $\widehat{F}$
in $S^3/Q(8)$.
The intersection of two transverse copies of $\widehat{F}$ is a meridian for the other component, and can be pushed off $\widehat{F}$.
Hence the intersection of three mutually transverse copies of $\widehat{F}$ is empty.

If $L$ is the link of Figure 3 then $M(L)$ is the ``half-turn" flat 3-manifold $MT(-I_2)$ 
with $\pi_1MT(-I_2)\cong\mathbb{Z}^2\rtimes_{-I}\mathbb{Z}$.
Poincar\'e duality and consideration of subgroups and quotient groups
shows that 
\[H^*(MT(-I_2);\mathbb{F}_2)\cong
\mathbb{F}_2[t,u,v]/(t^2, u^3, v^3, tu^2, tv^2, tuv+u^2v,u^2v+uv^2).
\]

\noindent {\it Realization:}
Let $\mathcal{A}^*$ be an orientable $MS$-algebra and let 
$\rho=\dim_{\mathbb{F}_2}\mathcal{A}^1$.
Then $u^2v=uv^2$ and so $(u+v)^3=u^3+v^3$,
for all $u,v\in\mathcal{A}^1$.
Let $\nu:\odot^3\mathcal{A}^1\to\mathcal{A}^3$ be the triple product.
If $\{x_1,\dots,x_\rho\}$ is a basis for $\mathcal{A}^1$,
we let $\nu_{ijk}=\nu(x_i\odot{x_j}\odot{x_k})$ for $1\leq{i,j,k}\leq\rho$.
Given  $u,v,z\in\mathcal{A}^1$,
there are 10 possible triple products involving just these elements,
but their values are constrained by the Postnikov-Wu identity,
and the number of possibilities to consider may be reduced further by 
judicious choice of basis.

If $x^2=0$ for all $x\in\mathcal{A}^1$ then we may model $\mathcal{A}^*$ by 
$M(L)$,
where $L$ is a  $\rho$-component link with all nontrivial 3-component sublinks 
copies of $Bo$.

Suppose next that $Sq^1x_i=x_i^2\not=0$ for $i\leq\sigma$, for some $\sigma>0$.
We may assume that $\{x_{\sigma+1},\dots,x_\rho\}$ spans $\mathrm{Ker}(Sq^1)$,
and so $Sq^1$ maps $X=\langle{x_1,\dots,x_\sigma}\rangle$ bijectively to
$Sq^1X=\langle{x_1^2,\dots,x_\sigma^2}\rangle$.
If  $x^3=0$ for all $x\in\mathcal{A}^1$ then
after a change of basis, if necessary, 
we may assume that the restricted pairing is block diagonal, 
with diagonal blocks $\left[\smallmatrix0&1\\1&0\endsmallmatrix\right]$.
For if $x_1^2y\not=0$ for some $y\in\mathcal{A}^1$ then $x_1y^2\not=0$,  
and we may set $x_2=y$.
We then modify the other basis elements $x_j$ with $j>2$ within their cosets
{\it mod\/} $\langle{x_1,x_2}\rangle$ so that  $x_1^2x_j=x_2^2x_j=0$.
We continue by induction on $k$ to modify the $x_{2j-1}$ and $x_{2j}$
with $j>k$  so that $x_{2j-1}^2x_{2j}\not=0$
and $x_i^2x_{2j-1}=x_i^2x_{2j}=0$ for all $i\leq2k$.
In particular,  $\sigma$ is even.

We start with a $\rho$-component link which splits as a union of $\frac\sigma2$ 
copies of $L_{2,4}$ and a trivial link with $\rho-\sigma$ components,
and construct the desired link $L$ by modifying some trivial 3-tangles, 
as in Figure 1.$(b)$.
The generators of $\mathcal{A}^1$ correspond 
to the Kronecker duals of the meridians of $L$.
If $\nu_{ijk}\not=0$ for some $i<j<k$ and $x_i^2x_j=x_i^2x_k=x_j^2x_k=0$ 
then we arrange for $L_i\cup{L_j}\cup{L_k}$ to be a copy of $Bo$.
However if (say) $x_i^2x_j\not=0$ then we use instead the link of Figure 3,
in which the components $L_i$ and $L_j$ are linked.
We give all components framing 0.

\setlength{\unitlength}{1mm}
\begin{picture}(95,70)(-25,-18)

\put(5,20){\line(0,1){15}}
\qbezier(5,20)(5,10)(15,10)
\qbezier(5,35)(5,45)(15,45)
\put(15,45){\line(1,0){10}}
\qbezier(25,45)(35,45)(35,35)

\put(34,34){\line(1,0){22}}
\qbezier(34,34)(32,34)(32,32)
\qbezier(32,32)(32,30)(34,30)
\put(35,27){\line(0,1){6}}
\qbezier(36,30)(38,30)(38,28)
\qbezier(38,28)(38,26)(36,26)
\put(34,26){\line(1,0){2}}
\qbezier(34,26)(32,26)(32,24)
\qbezier(32,24)(32,22)(34,22)
\put(35,20){\line(0,1){5}}
\qbezier(36,22)(39,22) (39,19)
\qbezier(39,17)(39,12)(44,12)
\put(44,12){\line(1,0){10}}
\qbezier(25,10)(35,10)(35,20)
\put(22,10){\line(1,0){4}}

\put(17,5){\line(0,1){3}}
\qbezier(17,8)(17,18)(27,18)
\put(27,18){\line(1,0){6.5}}
\qbezier(17,5),(17,-5)(27,-5)
\put(27,-5){\line(1,0){18}}

\put(36,18){\line(1,0){5}}
\qbezier(41,18)(42.5,18)(42.5,16.5)
\qbezier(41,15)(42.5,15)(42.5,16.5)

\qbezier(20,7.5)(20,15)(27.5,15)
\put(27.5,15){\line(1,0){5.5}}
\put(35,15){\line(1,0){3}}
\put(20,5){\line(0,1){2.5}}
\qbezier(20, 5)(20,-2)(27,-2)
\put(27,-2){\line(1,0){18}}
\qbezier(45,-2)(52,-2)(52,5)
\put(52,5){\line(0,1){6}}
\put(55,5){\line(0,1){9}}
\qbezier(55,14)(55,15.5)(53.5,15.5)
\qbezier(52,14)(52,15.5)(53.5,15.5)
\put(52,13){\line(0,1){1}}

\qbezier(45,-5)(55,-5)(55,5)
\qbezier(56,12)(66,12)(66,22)
\qbezier(56,34)(66,34)(66,24)
\put(66,22){\line(0,1){2}}

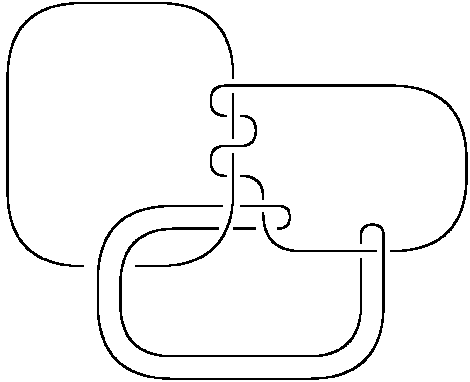
\put(10,-13){Figure 3.  Linking numbers 0,  0, 2.}

\end{picture}

If there is an $x$ such that $x^3\not=0$ then we may assume that $x_1^3\not=0$.
After replacing $x_i$ by $x_i+x_1$,  if necessary, 
we may assume that $x_1^2x_i=0$ for all $i>1$.
Hence the only possible nonzero triple products involving $x_1$ are of the form
$xyz$ with $x,y,z$ linearly independent.
In order to achieve this we choose our link so that all linking numbers 
$\ell(L_1,L_i)$ with $i>1$ are 0,
and the framing of $L_1$ is 2.
In this case $\mathrm{Ker}(Sq^1)$ has odd codimension.
The construction proceeds as before.
(With the above choices the multiplication pairing between $X$ and $Sq^1X$ 
is again block diagonal,   but there may be several basis elements $x$ with $x^3\not=0$,
and so the diagonal blocks are $[1]$ and 
$\left[\smallmatrix0&1\\1&0\endsmallmatrix\right]$.)

\noindent{\it Examples.}
If $M$ is orientable and $\rho=\beta_1(M;\mathbb{F}_2)=1$ then 
$H^*(M;\mathbb{F}_2)\cong\mathbb{F}_2[x]/(x^4)$ or $\mathbb{F}_2[x,u]/(x^2, u^2)$, 
with $x\in{H^1(M;\mathbb{F}_2)}$ and $u\in{H^2(M;\mathbb{F}_2)}$.
The simplest examples are $\mathbb{RP}^3=L(2,1)$, 
$S^1\times{S^2}$ and $L(4,1)$.

It follows from the Postnikov-Wu identity that $x^3+y^3=(x+y)^3$
for all $x,y$ of degree 1.
Hence if $\{x,y\}$ is a basis for $H^1(M;\mathbb{F}_2)$ 
such that $x^3\not=0$ then we may assume that $y^3\not=0$ also.
Nonsingularity of Poincar\'e duality then implies that $xy=0$,
and so $x\odot{y}$ generates the kernel of cup product.
For example,  $\mathbb{RP}^3\#\mathbb{RP}^3$.
Otherwise,  if $\{x,y\}$ is a basis for $H^1(M;\mathbb{F}_2)$ and $z^3=0$ 
for all $z$ of degree 1 then $x^2y=xy^2\not=0$,
and the kernel of cup product is generated by $x\odot{x}+x\odot{y}+y\odot{y}$, 
as for the quaternion manifold $S^3/Q(8)$.

Two more 3-component links shall play a role in the nonorientable case.

Let $L$ be the link of Figure 4, and let $A,P$ and $X$ be the classes in 
$H_1(M(L);\mathbb{Z})$ represented by the meridians $a,p$ and $x$.
Then $2(A+P)=2X=0$.
Let $\{u,v,z\}$ be the basis for $H^1(M(L);\mathbb{F}_2)$ which is Kronecker dual to the basis for $H_1(M(L);\mathbb{F}_2)$ represented by the meridians
(so that $u=a^*$, $v=p^*$ and $z=x^*$).
Then $u^3=v^3=z^3=0$ (since the components have framing 0),
while $z^2\not=0$ (since $2X=0$).
Let $\langle\langle{x}\rangle\rangle$ be the normal closure of the image of $x$ in 
$\pi_1M(L)$. 
Then $G=\pi_1M(L)/\langle\langle{x}\rangle\rangle$
has presentation $\langle{a,p}\mid(ap^{-1})^2=1\rangle$.
Then $u^2=uv=v^2$, since these relations hold in $H^*(G;\mathbb{F}_2)$,
and so  $u^2v=uv^2=u^3=0$.
Hence $uz^2=vz^2=uvz\not=0$, by nonsingularity, and so
\[
H^*(M(L);\mathbb{F}_2)\cong\mathbb{F}_2[u,v,z]/(u^2+uv,v^2+uv, u^2v,
u^2z+uz^2,v^2z+vz^2,z^3).
\]

\setlength{\unitlength}{1mm}
\begin{picture}(95,66)(-25,-18)

\put(5,20){\line(0,1){15}}
\qbezier(5,20)(5,10)(15,10)
\qbezier(5,35)(5,45)(15,45)
\put(15,45){\line(1,0){10}}
\qbezier(25,45)(35,45)(35,35)

\put(26,7){$a$}
\put(22,9){$\to$}

\put(47,8){$x$}
\put(44,11){$\to$}

\put(35,-6){$\to$}
\put(56,0){$p$}

\put(34,34){\line(1,0){22}}
\qbezier(34,34)(32,34)(32,32)
\qbezier(32,32)(32,30)(34,30)
\put(35,27){\line(0,1){6}}
\qbezier(36,30)(38,30)(38,28)
\qbezier(38,28)(38,26)(36,26)
\put(34,26){\line(1,0){2}}
\qbezier(34,26)(32,26)(32,24)
\qbezier(32,24)(32,22)(34,22)
\put(35,20){\line(0,1){5}}
\qbezier(36,22)(39,22) (39,19)
\qbezier(39,17)(39,12)(44,12)
\put(44,12){\line(1,0){7}}
\put(53,12){\line(1,0){3}}
\qbezier(25,10)(35,10)(35,20)
\put(22,10){\line(1,0){4}}

\put(17,5){\line(0,1){3}}
\qbezier(17,8)(17,18)(27,18)
\put(27,18){\line(1,0){6.5}}
\qbezier(17,5),(17,-5)(27,-5)
\put(27,-5){\line(1,0){18}}

\put(36,18){\line(1,0){5}}
\qbezier(41,18)(42.5,18)(42.5,16.5)
\qbezier(41,15)(42.5,15)(42.5,16.5)

\qbezier(20,7.5)(20,15)(27.5,15)
\put(27.5,15){\line(1,0){5.5}}
\put(35,15){\line(1,0){3}}
\put(20,5){\line(0,1){2.5}}
\qbezier(20, 5)(20,-2)(27,-2)
\put(27,-2){\line(1,0){18}}
\qbezier(45,-2)(52,-2)(52,5)
\put(52,5){\line(0,1){9}}
\put(55,5){\line(0,1){6}}
\qbezier(55,14)(55,15.5)(53.5,15.5)
\qbezier(52,14)(52,15.5)(53.5,15.5)

\qbezier(45,-5)(55,-5)(55,5)
\qbezier(56,12)(66,12)(66,22)
\qbezier(56,34)(66,34)(66,24)
\put(66,22){\line(0,1){2}}

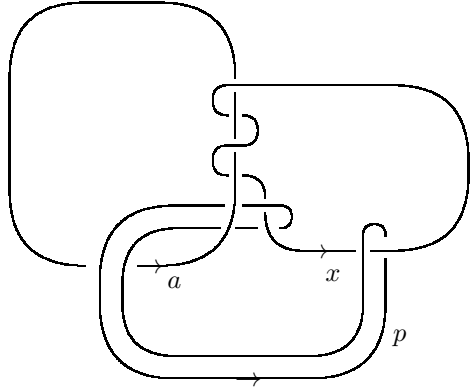
\put(5,-13){Figure 4.  Linking numbers 0,  2, 2.}

\end{picture}

The link of Figure 5 has a 3-fold symmetry (about the axis through $\bullet$)
which permutes the components.
If $A,P$ and $X$ are the classes in $H_1(M(L);\mathbb{Z})$ represented by 
the  meridians then $2A+2P=2A+2X=2P+2X=0$.
Hence $H_1(M(L);\mathbb{Z})\cong
\mathbb{Z}/4\mathbb{Z}\oplus\mathbb{Z}/2\mathbb{Z}\oplus\mathbb{Z}/2\mathbb{Z}$.
Let $\{u,v,z\}$ be the basis for $H^1(M(L);\mathbb{F}_2)$ which is Kronecker dual to 
the basis for $H_1(M(L);\mathbb{F}_2)$ represented by $\{a,p,x\}$.
Then $u^2$, $v^2$ and $z^2$ are all nonzero, with the sole relation $u^2+v^2+z^2=0$
(since $u+v+z$ lifts to an epimorphism onto $\mathbb{Z}/4\mathbb{Z}$).
It is then a straightforward exercise using nonsingularity and the 3-fold symmetry
to show that $u^2v=u^2z=v^2z=uvz\not=0$, and so
$H^*(M(L);\mathbb{F}_2)\cong\mathbb{F}_2[u,v,z]/\mathcal{I}$, where
\[
\mathcal{I}=(u^2+v^2+z^2, u^2v+uv^2,u^2z+uz^2,  
v^2z+vz^2,u^2v+uvz, v^2z+uvz).
\]

\setlength{\unitlength}{1mm}
\begin{picture}(90,70)(-15,0)

\put(20,46){$a$}
\put(41,16){$p$}
\put(70,46){$z$}

\put(45,39){$\bullet$}

\put(22.1,40){$\downarrow$}
\put(67.1,40){$\uparrow$}
\put(46,19.15){$\to$}

\put(26,20){\line(1,0){29}}
\put(28,32){\line(1,0){6}}
\put(36,32){\line(1,0){22}}

\qbezier(23,57)(23,63)(29,63)
\qbezier(29,63)(35,63)(35,57)
\qbezier(20,26)(20,32)(26,32)
\qbezier(20,26)(20,20)(26,20)

\put(23,33){\line(0,1){24}}
\put(35,49){\line(0,1){6}}
\put(35,30){\line(0,1){17}}
\put(27,30){\line(0,1){3}}

\qbezier(23,30)(23,28)(25,28)
\qbezier(25,28)(27,28)(27,30)
\qbezier(27,33)(27,35)(29,35)
\qbezier(29,35)(31,35)(31,33)

\qbezier(31,30)(31,28)(33,28)
\qbezier(33,28)(35,28)(35,30)

\put(36,60){\line(1,0){20}}
\put(33,48){\line(1,0){19}}

\qbezier(52,48)(56,48)(56,44)
\qbezier(56,60)(68,60)(68,48)

\put(56,33){\line(0,1){11}}
\put(68,23){\line(0,1){25}}
\qbezier(62,17)(68,17)(68,23)
\qbezier(56,23)(56,17)(62,17)

\put(56,25){\line(0,1){6}}

\put(57,28){\line(1,0){1}}
\qbezier(58,20)(60,20)(60,22)
\qbezier(58,24)(60,24)(60,22)
\qbezier(58,32)(60,32)(60,30)
\qbezier(58,28)(60,28)(60,30)
\qbezier(55,24)(53,24)(53,26)
\qbezier(53,26)(53,28)(55,28)
\put(55,24){\line(1,0){3}}

\qbezier(33,60)(31,60)(31,58)
\qbezier(31,58)(31,56)(33,56)
\put(33,56){\line(1,0){3}}
\qbezier(36,56)(38,56)(38,54)
\qbezier(36,52)(38,52)(38,54)
\qbezier(33,52)(31,52)(31,50)
\qbezier(31,50)(31,48)(33,48)

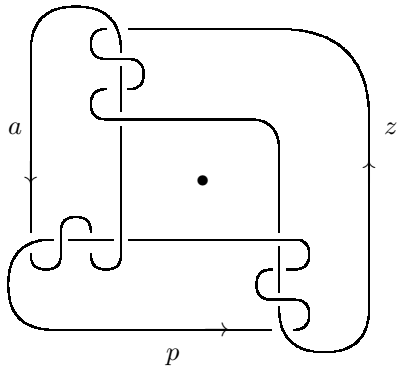
\put(23,5){Figure 5. All linking numbers 2.}

\end{picture}

\section{links in $S^2\tilde\times{S^1}$}

Every closed nonorientable 3-manifold may be obtained by
surgery on a framed link in (either of) $S^2\tilde\times{S^1}$ or 
$\mathbb{RP}^2\times{S^1}$ \cite{Li63}.
We shall construct links in $S^2\tilde\times{S^1}$ from tangles in $D^2\times[0,1]$
with endpoints on the discs $D^2\times\{0\}$ and $D^2\times\{1\}$ by identifying these discs via reflection across a diameter of $D^2$ to get a link in $Z=D^2\tilde\times{S^1}$,
and then attaching another copy of $Z$ to get the double
$DZ=Z\cup_{Kb}Z=S^2\tilde\times{S^1}$.
We shall assume that the endpoints of the tangle are paired under the reflection, 
and lie along the diameter fixed by the reflection.
(It is not hard to see that every link in $S^2\tilde\times{S^1}$ arises in this way,
but we shall not need to prove this.)

\setlength{\unitlength}{1mm}
\begin{picture}(95,57)(-15,25)

\put(1,67.5){$a_1$}
\put(1,60.5){$a_2$}
\put(-12.5,58){$\omega$}
\put(-9,59){$\downarrow$}

\linethickness{1pt}
\put(7,70){\line(0,1){5}}
\qbezier(7,75)(12, 69)(7,63)
\put(87,70){\line(0,1){5}}

\thinlines
\qbezier(-8,60)(-8,75)(7,75)
\qbezier(-8,60)(-8,45)(7,45)
\qbezier(7,45)(22,45)(22,60)
\qbezier(7,75)(22,75)(22,60)
\put(7,45){\line(0,1){15}}

\put(7,63){\line(0,1){4}}

\put(6.1,74){$\bullet$}
\put(6.5,61.4){.}
\put(6.5, 68.4){.}

\qbezier(5.5,61.5)(5.5,63)(7,63)
\qbezier(5.5,61.5)(5.5,60)(7,60)
\qbezier(7,63)(8.5,63)(8.5,61.5)
\qbezier(7,60)(8.5,60)(8.5,61.5)

\qbezier(5.5,68.5)(5.5,70)(7,70)
\qbezier(5.5,68.5)(5.5,67)(7,67)
\qbezier(7,70)(8.5,70)(8.5,68.5)
\qbezier(7,67)(8.5,67)(8.5,68.5)

\qbezier(72,60)(72,75)(87,75)
\qbezier(72,60)(72,45)(87,45)
\qbezier(87,45)(102,45)(102,60)
\qbezier(87,75)(102,75)(102,60)
\put(87,45){\line(0,1){15}}

\put(87,63){\line(0,1){4}}

\put(86,74){$\bullet$}
\put(86.5,61.4){.}
\put(86.5, 68.4){.}

\qbezier(85.5,61.5)(85.5,63)(87,63)
\qbezier(85.5,61.5)(85.5,60)(87,60)
\qbezier(87,63)(88.5,63)(88.5,61.5)
\qbezier(87,60)(88.5,60)(88.5,61.5)

\qbezier(85.5,68.5)(85.5,70)(87,70)
\qbezier(85.5,68.5)(85.5,67)(87,67)
\qbezier(87,70)(88.5,70)(88.5,68.5)
\qbezier(87,67)(88.5,67)(88.5,68.5)

\qbezier(87,75)(82, 69)(87,63)

\put(64,58){$\omega^{-1}$}
\put(71.1,59){$\uparrow$}

\put(40,60){$\longmapsto$}

\put(0,37){$\omega=a_1a_2$}

\put(68,37){$\omega^{-1}=a_1^{-1}(a_1a_2^{-1}a_1^{-1})$}

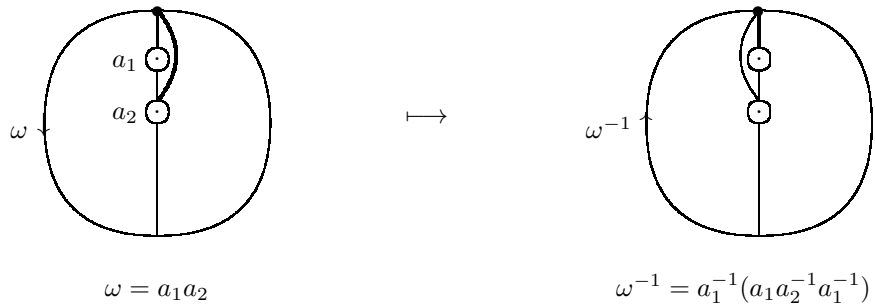
\put(15,30){Figure 6. Reflection across the $Y$-axis}

\end{picture}

Let $L$ be a link in $Z$ with orientation preserving components.
Then each component has an even number of endpoints at each end of $D^2\times[0,1]$,
and so bounds a surface in $Z$ which is a union of a spanning surface for 
a knot or link in $D^2\times[0,1]$ with a number of twisted ribbons. 
Moreover, 
we may assume that the twisted ribbons for the various components are all disjoint.
Although there is no canonical choice of longitude, we may choose ``longitudes"
which are $\mathbb{F}_2$-homologically trivial in $Z$,
and we shall let $Y(L)$ be the result of  surgery on $L$ in $Z$ to kill these longitudes.
Let $\Xi(L)=Y(L)\cup{Z}$ be the closed 3-manifold obtained by surgery on $L$ in $DZ$.
We shall take the core of the second copy of $Z$ as the standard orientation-reversing loop in $\Xi(L)$.

We may write down a presentation for the fundamental group of the link exterior 
in terms of meridians and ``Wirtinger" relations in the usual way, 
but there is a slight complication due to the reflection.
We assume that the fixed diameter of $D^2$ is the intersection with the $Y$-axis, 
and take the top of this diameter in $D^2\times\{0\}$ as the basepoint.
The meridian for each arc $\alpha$ of the tangle is represented by a loop from the basepoint which passes ``in front of" all intermediate arcs, around $\alpha$, 
and back over the other arcs to the basepoint.
The effect of the reflection is evident from Figure 6,
in which the loops $a_1$ and $a_2$ are based via the heavy curves, 
and $a_1$, $a_2$ and $\omega$ go anticlockwise around the adjacent circles.
(Similarly, if $a_1,\dots,a_n$ are generators corresponding to $n$ punctures down the diameter then $\omega=a_1\dots{a_n}$ and 
$\omega^{-1}=a_1^{-1}(a_1a_2^{-1}a_1^{-1})\dots(\omega{a_n}^{-1}\omega^{-1})$.)

We shall illustrate this by giving constructions of $S^1\times\mathbb{RP}^2$
and $S^1\times{Kb}$.
The fundamental group of the complement of the tangle in Figure 7 
has a presentation $\langle{a,b,c}\mid{aba^{-1}=c}\rangle$.
Identifying the ends gives a knot in $Z$ whose exterior
has fundamental group with the presentation
\[
\langle{a,b,c,t}\mid{aba^{-1}=c},~tat^{-1}=c^{-1},~tbt^{-1}=ca^{-1}c^{-1}\rangle.
\]
Adding another copy of $Z$ gives a knot in $S^2\tilde\times{S^1}$
whose group has presentation
\[
\langle{a,b,c,t}\mid{aba^{-1}=c},~tat^{-1}=c^{-1},~tbt^{-1}=ca^{-1}c^{-1},~ab=1\rangle,
\]
since $\omega=ab$ bounds a fibre of the second copy of $D^2\tilde\times{S^1}$.
If we now perform surgery on the knot to kill $ta^{-1}ta$ 
the presentation reduces to $\langle{a,t}\mid{ta=at},~t^2=1\rangle$,
and so the fundamental group is
$\mathbb{Z}\oplus(\mathbb{Z}/2\mathbb{Z})$.

\setlength{\unitlength}{1mm}
\begin{picture}(95,50)(-15,30)

\linethickness{1pt}
\qbezier(2,60)(2,75)(17,75)
\qbezier(2,60)(2,45)(17,45)
\qbezier(17,45)(32,45)(32,60)
\qbezier(17,75)(32,75)(32,60)

\qbezier(77,45)(92,45)(92,60)
\qbezier(77,75)(92,75)(92,60)

\thinlines
\put(16,74){$\bullet$}
\put(16.5,54.4){.}
\put(16.5, 64.4){.}
\put(76.5,54.4){.}
\put(76.5, 64.4){.}

\put(23,63.8){$\to$}
\put(23,53.8){$\to$}
\put(70,63.8){$\to$}
\put(70,53.8){$\to$}
\put(37,66){$a$}
\put(37,56){$b$}
\put(67,66){$c$}

\put(17,75){\line(1,0){60}}
\put(17,45){\line(1,0){60}}

\put(17,54.7){\line(1,0){13}}
\put(33,54.7){\line(1,0){8}}
\qbezier(41,54.7)(47,54.7)(52,58)
\qbezier(54,60)(59,64.7)(67,64.7)

\put(17,64.7){\line(1,0){13}}
\put(33,64.7){\line(1,0){8}}
\qbezier(41,64.7)(46,64.7)(49,62)
\qbezier(49,62)(53,59)(57,57)
\qbezier(57,57)(61,54.7)(65,54.7)

\put(65,54.7){\line(1,0){12}}

\put(67,64.7){\line(1,0){10}}

\put(21,35){Figure 7.  A link for $S^1\times\mathbb{RP}^2$}

\end{picture}

From another point of view, we may see directly that surgery on 
$U_o=\{*\}\times{S^1}\subset{S^1}\times\mathbb{RP}^2$
gives $S^2\tilde\times{S^1}$,
as follows.
Let $\theta$ be the reflection of $S^2$ across an equator.
Then $S^2\tilde\times{S^1}\cong{MT(\theta)}$, 
the mapping torus of $\theta$.  
Since $\theta$ swaps the ``polar zones" of $S^2$, the mapping torus of the 
restriction of $\theta$ to the union of the polar zones is a solid torus $D^2\times{S^1}$
(representing $2\times$ a generator of $H_1(S^2\tilde\times{S^1};\mathbb{Z})$).
The complement of the polar zones is an annulus $\mathbb{A}\cong{S^1}\times[-1,1]$,
and $\theta$ acts on this product via $\theta(s,x)=(s,-x)$,
for all $s\in{S^1}$ and $-1\leq{x}\leq1$. 
Hence $MT(\theta|_\mathbb{A})\cong{S^1}\times{Mb}$.
Since $S^1\times\mathbb{RP}^2=S^1\times{Mb}\cup{S^1}\times{D^2}$
and $S^2\tilde\times{S^1}=MT(\theta|_\mathbb{A})\cup{D^2}\times{S^1}$,
the claim follows.
Note also that $H_1(S^2\tilde\times{S^1};\mathbb{Z})$ is generated 
by a meridian for $U_o$.

The fundamental group of the complement of the tangle in Figure 7 has a presentation
$\langle{a,c,q}\mid{aqa^{-1}=cqc^{-1}}\rangle$.
Identifying the ends gives a link in $Z$ whose exterior
has fundamental group with the presentation
\[
\langle{a,b,c,d,q,t}\mid {aqa^{-1}=cqc^{-1}},~qaq^{-1}=b,~qcq^{-1}=d,
\]
\[
tat^{-1}=c^{-1},~tbt^{-1}=cd^{-1}c^{-1}\rangle.
\]
In this case we may add another copy of $Z$ so that
$\omega=ab^{-1}$ bounds a disc.
Adding the relation $ab^{-1}=1$ and then killing the longitudes
$\ell_q=a^{-1}c$ and $\ell_a=q^{-1}t^{-1}qt$ gives the presentation
\[\langle{a,q,t}\mid{tat^{-1}=a^{-1}},~aq=qa,~qt=tq\rangle.
\]
Hence the resulting 3-manifold is $S^1\times{Kb}$.

If instead we give each of the link components  a nonzero even framing 
(so that we kill $q^{-1}t^{-1}qta^{2k}$ and $a^{-1}cq^{2m}$) then 
the resulting group is a semidirect product $\mathbb{Z}^2\rtimes_A\mathbb{Z}$,
where $A=\left[\smallmatrix-1&-2k\\2m&4km+1\endsmallmatrix\right]$.
This is the group of a nonorientable $\mathbb{S}ol^3$-manifold 
$MT(A)$ if $km\not=0$.
Taking $k=m=1$ gives an example with $\rho=3$, 
$w^2=0$ and $u^3=v^3=wuv\not=0$.
(See \cite[page 198]{Hi14}.
In this reference $\rho$ is the orientation character, called $w$ here.)

\setlength{\unitlength}{1mm}
\begin{picture}(95,52)(-15,30)

\linethickness{1pt}
\qbezier(2,60)(2,75)(17,75)
\qbezier(2,60)(2,45)(17,45)
\qbezier(17,45)(32,45)(32,60)
\qbezier(17,75)(32,75)(32,60)

\qbezier(77,45)(92,45)(92,60)
\qbezier(77,75)(92,75)(92,60)

\thinlines
\put(16,74){$\bullet$}
\put(16.5,54.4){.}
\put(16.5, 64.4){.}

\put(21,63.8){$\to$}
\put(23,53.8){$\gets$}

\put(70,63.8){$\to$}
\put(72,53.8){$\gets$}

\put(53,57){$\gets$}

\put(35,66){$a$}
\put(35,56){$b$}
\put(68,66){$c$}
\put(68,56){$d$}

\put(50,55){$q$}

\put(17,75){\line(1,0){60}}
\put(17,45){\line(1,0){60}}

\put(76.5,54.4){.}
\put(76.5, 64.4){.}

\put(17,54.7){\line(1,0){13}}
\put(33,54.7){\line(1,0){6}}
\qbezier(39,54.7)(43,54.7)(43,57)

\put(17,64.7){\line(1,0){13}}
\put(33,64.7){\line(1,0){6}}
\qbezier(39,64.7)(43.7,64.7)(43.7,60.7)

\put(67,54.7){\line(1,0){10}}
\qbezier(67,54.7)(63,54.7)(63,57)

\put(66,64.7){\line(1,0){11}}
\qbezier(66,64.7)(62,64.7)(62,60.7)

\put(42,58){\line(1,0){22}}
\qbezier(40,60)(40,58)(42,58)
\qbezier(40,60)(40,62)(42,62)
\qbezier(64,58)(66,58)(66,60)
\qbezier(66,60)(66, 62)(64,62)
\put(45,62){\line(1,0){16}}

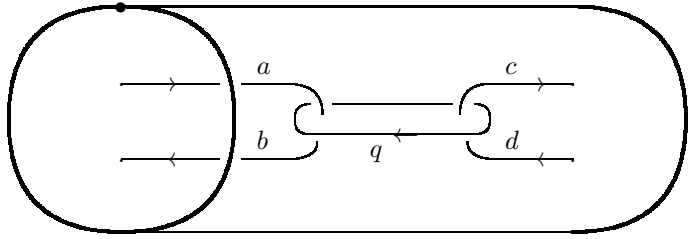
\put(10,35){Figure 8.  A link for $S^1\times{Kb}$ and other $MT(A)$s}

\end{picture}

\section{the nonorientable case}

Let $\mathcal{A}^*$ be an $MS$-algebra with distinguished element $w$
and rank $\rho=\dim\mathcal{A}^1$.
The Postnikov-Wu identity is now 
$wxy =x^2y+xy^2\quad\mathrm{for~all}~x,y\in\mathcal{A}^1$.
Hence $wx^2=0$ for all $x$ of degree 1.
In particular, $w^3=0$.

We shall show that  $\mathcal{A}^*$ is the $\mathbb{F}_2$-cohomology ring 
of a closed 3-manifold with orientation character corresponding to $w$.
In our constructions we shall reserve $w$ for the orientation character.
In examples with $\rho$ small Poincar\'e duality considerations may 
imply that some products $x^2$ or $xy$ are 0.
This is not a hindrance.
However we should consider the possible values of cubes $x^3$.

If $wx=0$ for all $x\in\mathcal{A}^1$ then $w^2=0$,
and $\mathcal{A}^*$ may be realized by the connected sum
of  $S^2\tilde\times{S^1}$ with an orientable closed 3-manifold.
(More generally, if $x\in\mathcal{A}^1$ is nonzero and $xy=0$
for all $y\in\mathcal{A}^1$ then it suffices to consider a subalgebra of rank $\rho-1$.)
Thus we may assume that $w\bullet:\mathcal{A}^1\to\mathcal{A}^2$ is nontrivial.

We begin by identifying the smallest $MS$-algebras of interest,
which have $\rho=2$ or 3 and $w\bullet\not=0$.
Since $w\not=0$ there is a basis $\{w=x_1,x_2,\dots,x_\rho\}$ for $\mathcal{A}^1$.
We shall modify the basis elements $x_i$ with $i>1$ so as to simplify the multiplication scheme.
Let $\sigma$ be the rank of $w\bullet:\mathcal{A}^1\to\mathcal{A}^2$.
Then $\sigma>0$, since $w\bullet\not=0$.
If $w^2\not=0$ we may choose the basis so that 
$\{w^2,wx_2,\dots,wx_\sigma\}$ are linearly independent and $wx_i=0$ for all $i>\sigma$.
If $w^2=0$ we assume that $\{wx_2,\dots,wx_{\sigma+1}\}$ are linearly independent 
and $wx_i=0$ for all $i>\sigma+1$.

If $wu\not=0$ then there is a $v$ such that $wuv\not=0$, by nonsingularity.
The elements $wu$ and $wv$ must be linearly independent, since $wu^2=0$.
If $w^2\not=0$ we may assume that $w^2x_2\not=0$.
After replacing $x_i$ by $x_i+w$ or $x_i+w+x_2$, if necessary,
we may assume that $w^2x_i=wx_2x_i=0$ for all $i>2$.
If $\sigma>2$ we may assume that $wx_3x_4\not=0$.
Further modifications to the basis elements $x_i$ with $i>4$ 
then ensure that $wx_3x_i=wx_4x_i=0$ for all $i>4$.
Iterating this process, we see that $\sigma$ must be even, 
and the partial basis $\{w,x_2,\dots,x_\sigma\}$ is partitioned into consecutive pairs
whose triple products with $w$ are nonzero, and are the only such  products $wx_ix_j$.
A similar argument applies if $w^2=0$, and $\sigma$ is again even.
In particular,  if $w\bullet\not=0$ then $\rho>1$.
We may also choose the basis elements $x_i$ with $i>\sigma$ so that
$x_i^2x_j\not=0$ for at most one $j>i$, as in the orientable case,
but it is not clear that we can do this in general.

We shall realize these by closed 3-manifolds
obtained by surgery on framed links $L\subset{S^2\tilde\times{S^1}}$.
We may assume that such links are disjoint from a standard orientation 
reversing loop in $S^2\tilde\times{S^1}$. 
Let $B(L)$ be the complement of a small open tubular neighbourhood
of this loop  in the closed 3-manifold resulting from surgery on $L$.
Then $\partial{B(L)}\cong{Kb}$.
Let $N=\overline{S^2\setminus(\cup_{i=1}^k{D_i})}$,
where the $D_i$ are disjoint small discs in $S^2$ 
which are centred on the equator and invariant under $\theta$.
Then $Y=MT(\theta|_N)$ has boundary $\partial{Y}\cong\sqcup^kKb$,
and we may attach $k$ such 3-manifolds $B(L(i))$ to $Y$, 
along the various boundary components.
If $M$ is the resulting 3-manifold  and we specify an orientation-reversing curve
$\gamma_w$ then the generators of $H_1(M;\mathbb{F}_2)$
apart from the image of $\gamma_w$ may be chosen to be
Kronecker dual to the given basis for the $MS$-algebra.
These arise from loops which lift to the double cover 
$S^2\times{S^1}$ and have product neighbourhoods.
We may then meld in products involving the $x_i$ with $wx_i=0$ by adding 
copies of $Bo$ and $L_{2,4}$, as appropriate,  
to the existing link in $S^2\tilde\times{S^1}$.

There are four cases with $1\leq\rho\leq3$ that we need consider.
If $\rho=1$ then $w^2=0$, by nonsingularity.
This case is realized by $S^2\tilde\times{S^1}$.

If $w^2\not=0$ then after replacing $u=x_2$ by $u+w$, 
if necessary, we may assume that $u^3=0$.
This case is realized by $S^1\times\mathbb{RP}^2$.

If $w^2=0$ and $\rho=3$ then we may assume that $wuv\not=0$,
for some $u,v\in\mathcal{A}^1$.
If $u^3=0$ or $v^3=0$ then after replacing $v$ or $u$ by $u+v$, 
if necessary,  we may assume that $u^3=v^3=0$, and that $u^2v\not=0$.
This case is realized by $S^1\times{Kb}$.

The other basic case has $w^2=0$, $\rho=3$ and $wuv\not=0$,
and $u^3=v^3=(u+v)^3\not=0$.
This case is realized by the $\mathbb{S}ol^3$-manifold
$MT(\left[\smallmatrix-1&-2\\2&5\endsmallmatrix\right])$.

These  3-manifolds shall be our basic building blocks.
If $\Xi$ is one of these then in each case $H_1(\Xi;\mathbb{F}_2)$ 
has a preferred basis consisting of a ``standard" orientation-reversing loop 
and the images of meridians for the components of $L$, 
and there is a well-defined ``Kronecker dual" basis for $H^1(\Xi;\mathbb{F}_2)$.

\section{construction -- splicing}

In this section we shall show that every (nonorientable) $MS$-algebra is realized
by some (nonorientable) closed 3-manifold.
We begin by ``splicing" the examples of \S8 to realize $MS$-algebras for which 
$w\bullet$ is an isomorphism or has kernel $\langle{w}\rangle$.
To handle the general case we shall rely on the following simple observation.
Let $K$ be an orientable knot in a 3-manifold $M$, 
and let $N_{\mathcal{F}}$ be the result of surgery on $K$ in $M$ with framing 
$\mathcal{F}$.
Let $F_1$, $F_2$ and $F_3$ be mutually transverse closed surfaces in 
the knot exterior $X=M\setminus{K}$,
and let $x_i\in{H^1(N_{\mathcal{F}};\mathbb{F}_2)}$ be
the Poincar\'e dual of the class of $F_i$,
for $i\leq3$.
Then whether $x_1x_2x_3$ is 0 or not does not depend on the framing $\mathcal{F}$
determining the surgery, 
since any intersection of submanifolds in the exterior $X=M\setminus{K}$
is unchanged by any (Dehn) surgery on $K$.

We shall splice links in $S^2\tilde\times{S^1}$ together as follows.
Let $D$ and $D'$ be two small disjoint discs in the interior of $D^2$
which are invariant under the reflection across the $Y$-axis,  
and let $W=\overline{D^2\setminus{D\cup{D'}}}$.
Then $Z=W\tilde\times{S^1}\cup{D\tilde\times{S^1}}\cup{D'\tilde\times{S^1}}$.
Let $L$ and $L'$ be links with orientation-preserving components  in $D\tilde\times{S^1}$ and $D'\tilde\times{S^1}$, respectively.
Then $L$, $L'$ and $L\sqcup{L'}$ are links in $Z$, 
and there are degree-1 collapses from $\Xi(L\sqcup{L'})$ onto each of $\Xi(L)$ 
and $\Xi(L')$.
Hence $H^*(\Xi(L);\mathbb{F}_2)$ and $H^*(\Xi(L);\mathbb{F}_2)$ map injectively to
$H^*(\Xi(L\sqcup{L'});\mathbb{F}_2)$.
Since $H_1(\Xi(L);\mathbb{F}_2)$ is generated by the images of $t$ and the meridians for $L$ it follows that $H^*(\Xi(L\sqcup{L'});\mathbb{F}_2)$ is generated by the union
of the images of $H^*(\Xi(L);\mathbb{F}_2)$ and $H^*(\Xi(L);\mathbb{F}_2)$.
Each component of $L$ has image 0 in $H_1(D\tilde\times{S^1};\mathbb{F}_2)$, 
and so bounds a surface in $D\tilde\times{S^1}$.
It follows that the duals of a meridian from $L$ and one from $L'$ 
may be assumed disjoint, 
and so cup products of the corresponding basis elements for $H^1$ are 0.

If $w\bullet$ is an isomorphism then $w^2\not=0$ and
we may partition $\{x_1=w,\dots,x_\sigma\}$ into consecutive pairs
$\{x_{2i-1},x_{2i}\}$ with $wx_{2i-1}x_{2i}\not=0$ for all $i\leq\frac12\sigma$
and $wx_jx_k=0$ if $k>j+1$ or  $j$ is even and $k>j$.
(In the latter case we have $x_j^2x_k=x_jx_k^2$, 
by the Postnikov-Wu identity.)
The pair $\{w,x_2\}$ is realized  by $S^1\times\mathbb{RP}^2$,
while the other pairs are realized by $S^1\times{Kb}$ (if $x_{2i-1}^3=x_{2i}^3=0$) 
or $MT(\left[\smallmatrix-1&-2\\2&5\endsmallmatrix\right])$ 
(if $x_{2i-1}^3=x_{2i}^3\not=0)$.
Thus we may realize $\mathcal{A}^*$ in this case by assembling copies of the links of 
Figures 7 and 8 (with appropriate framings).

A similar argument applies if $\mathrm{Ker}(w\bullet)=\langle{w}\rangle$.
In this case we need just copies of the link of Figure 8 (with appropriate framings).

In general,  let $\tau=\rho-\sigma$ if $w^2\not=0$ and $\tau=\rho-\sigma-1$ if $w^2=0$,
and adjoin a trivial $\tau$-component link in a ball in $S^2\times(0,1)$ 
which is disjoint from the other components.
We must now consider the possibility that $x_ix_jx_k\not=0$, 
where $1<{i}\leq{j}\leq{k}$.
We may assume that $k>i$,  
since the value of $x_i^3$ is determined by the framing for $L_i$,
and shall write $x=x_i$,  $y=x_j$ and $z=x_k$ for simplicity of notation.

If $wxy=wxz=wyz=0$ then $x^2y=xy^2$, $x^2z=xz^2$ and $y^2z=yz^2$,
by the Postnikov-Wu identity, and one of the constructions for the orientable case applies.
However, having chosen the basis so as to normalize the nonzero products $wuv$,
we may not be able to reduce the number of possibilities for other triple products.
Thus we shall use tangles based on Figures 4 and 5 as well as those of Figures 1 and 3.

If  $x^2y=x^2z=y^2z=0$ then we use move $(b)$ of Figure 1.

If  $x^2z=y^2z=0$ but $x^2y\not=0$ then we use a tangle based on the link of Figure 3,
with $z$ corresponding to the component which is unlinked from each of the other components.

If  $x^2y=0$ but $x^2z=y^2z\not=0$ then we use a tangle based on the link of Figure 4,
with $z$ corresponding to the component which links each of the other components.

If all three of these products are nonzero then we use moves of type $(a)$ as in the link of Figure 5.

Thus we may assume that $wxy\not=0$.
In particular, $x\not=y$.
We may also assume that $z\not=x$ or $y$, 
and so $wxz=wyz=0$, by our choice of basis.
The Poincar\'e dual of $w$ in $S^2\tilde\times{S^1}$ 
is represented by a fibre $S^2$.
It is easily seen that this remains true after the tangle modifications used below,
and so these do not disrupt the values of $wxy$, $x^2y$ or $xy^2$.

If $x^2z=y^2z=0$ then we use a move $(b)$.

If $x^2z\not=0$ and $y^2z=0$ then we use a tangle based on the link of Figure 3.

If $x^2z=y^2z\not=0$ then we use move $(a)$ twice.


Finally, we choose the framings of the components lying entirely in $S^2\times(0,1)$
in accordance with the desired values of the $x_i^3$s.
 
\section{the kernel of cup product}

If $R=\mathbb{Z}$ or is a field of characteristic $\not=2$ then cup product 
induces homomorphisms $c_G^R:\wedge_2H^1(G;R)\to{H^2(G;R)}$.
We shall write $c_G$ and $c_G^p$ when $R=\mathbb{Z}$ or $\mathbb{F}_p$, 
respectively.
If $G$ is finitely generated then
\[
\mathrm{Ker}(c_G)\cong{Hom(I(G)/[G,I(G)],\mathbb{Z})}.
\]
Similarly, if $p$ is an odd prime then
\[
\mathrm{Ker}(c_G^p)\cong{Hom(G'X^p(G)/[G,G']X^p(G),\mathbb{F}_p)}.
\]
If $p=2$ then cup product induces  $c_G^2: \odot_2H^1(G;\mathbb{F}_2)\to
{H^2(G;\mathbb{F}_2)}$ with
\[
\mathrm{Ker}(c_G^2)\cong{Hom(X^2(G)/[G,X^2(G)]X^4(G),\mathbb{F}_2)}.
\]
In all cases the kernel is determined by $G/\gamma_3G=G/[G,G']$.
See  \cite{Hi85, Hi87, Li91} for proofs of the above assertions.
(In \cite{Su75} $\mathrm{Ker}(c_G)$  is said to be isomorphic to 
$Hom(G'/[G,G'],\mathbb{Z})$ ``mod torsion". 
In fact $I(G)/[G,I(G)]$ and $G'/[G,G']$ are commensurable.)

\begin{lemma}
Let $M$ be an orientable $3$-manifold, let $\pi=\pi_1(M)$
and let $F$ be a field of characteristic $\not=2$.
Then
\begin{enumerate}
\item{}if $\beta_1(\pi;F)<3$ then $c_\pi^F=0$;
\item{}if $\beta_1(\pi;F)=3$ then $c_\pi^F$ is either $0$ or is an isomorphism.
\item{}if $\beta_1(\pi;F)>3$ then $\mathrm{Ker}(c_\pi^F)\not=0$.
\end{enumerate}
\end{lemma}

\begin{proof}
If $\alpha\smile\xi\not=0$ for some $\alpha, \xi\in{H^1(M;F)}$
then $\alpha\smile\xi\smile\omega\not=0$, 
for some $\omega\in{H^1(M;F)}$,
by the nonsingularity of Poincar\'e duality.
Since $\alpha, \xi$ and $\omega$ must then be linearly independent, 
$\beta_1(M;F)\geq3$.

Suppose that $\beta_1(\pi;F)=3$.
Then every element of $\wedge_2H^1(M;F)$ is a product $v\wedge{w}$,
for if $a\not=0$ then 
$ax\wedge{y}+bx\wedge{z}+cy\wedge{z}=a^{-1}(ax-cz)\wedge(ay+bz)$.
Hence if $c_\pi^F$ is not an isomorphism then we may assume that
$\alpha\smile\xi=0$, where $\{\alpha, \xi,\omega\}$ is a basis for $H^1(M;F)$.
But then $\{\alpha\wedge\xi, \alpha\wedge\omega,\xi\wedge\omega\}$ 
is a basis for $\wedge_2H^1(M;F)$, and $\alpha\smile\xi\smile\omega=0$.
It  follows easily from the nonsingularity of Poincar\'e duality that $c_\pi^F=0$.

Let $\beta=\beta_1(M;F)$. Then $\beta_2(M;F)=\beta$ also, by Poincar\'e duality.
Hence $\dim_F\mathrm{Ker}(c_\pi^F)\geq\binom\beta2-\beta$,
and so $\mathrm{Ker}(c_\pi^F)\not=0$ if $\beta>3$.
\end{proof}

If $\pi\cong\mathbb{Z}^3$ then $\beta_1(M;\mathbb{Q})=3$ and 
$c_\pi^\mathbb{Q}$ is an isomorphism, 
while if $\pi\cong{F(3)}$ then $\beta_1(M;\mathbb{Q})=3$ and $c_\pi^\mathbb{Q}=0$.

The case $p=2$ is different.
If $\pi\cong\mathbb{Z}/2\mathbb{Z}$ then 
$\beta_1(M;\mathbb{F}_2)=1$ and $c_\pi^2$ is an isomorphism.
On the other hand,  if $\beta=\beta_1(M;\mathbb{F}_2)>1$ then
$\dim_F\mathrm{Ker}(c_\pi^2)\geq\binom{\beta+1}2-\beta>0$,
and so $\mathrm{Ker}(c_\pi^F)\not=0$.

Lemma 1 also has implications for the integral case.
If $\beta\leq2$ then $c$ must have image in $tH$,
while if $\beta=3$ then either $c$ has image in $tH$ or it maps 
$\wedge_1H^*$ onto $H/tH$.

\section{cup product and universal coefficients}

Let $p$ be a prime.
The image of $H^1(G;\mathbb{Z})$ in $H^1(G;\mathbb{F}_p)$ is canonical, 
and the restriction of $c_G^p$  to this image  is the {\it mod}-$p$
reduction of $c_G$.
However if $G^{ab}$ has $p$-torsion this does not fully determine $c_G^p$.

We may construct 3-manifold examples illustrating this as follows.
Let $M=M(Bo;p)$ and let $\pi=\pi_1(M)$.
Then $\pi^{ab}\cong(\mathbb{Z}/p\mathbb{Z})^3$ and $X^p(\pi)=\pi'$. 
Hence $c_\pi^p$ is injective, by the criterion of \cite{Hi85}.
On the other hand, if $N=\#^3L(p,q_i)$ (for some $q_i$ such that $(q_i,p)=1$)
and $G=\pi_1(N)\cong*^3(\mathbb{Z}/p\mathbb{Z})$ then $c_G^p=0$.
Clearly $H^*(M;\mathbb{Z})\cong{H^*(N;\mathbb{Z})}$ as rings.
Moreover, we may choose the parameters $q_i$ so that $\ell_M\cong\ell_N$.
Thus $H$,  $\gamma$, $D_2$ and $\ell$ do not determine the mod-$p$ cohomology ring.

The group $\Gamma_q$ with presentation 
\[
\langle{x,y,z}\mid [x,y]=z^q,~zx=xz,~zy=yz\rangle.
\]
is the fundamental group of a $\mathbb{N}il^3$-manifold,
and $\beta_1(\Gamma_q)=2$, 
so $c_{\Gamma_q}^\mathbb{Q}=0$, by Lemma 1.
On the other hand, 
$I(\Gamma_q)=\zeta\Gamma_q\cong\mathbb{Z}$ 
is generated by the image of $z$.
Hence $[\Gamma_q,I(\Gamma_q)]=1$,
but $\Gamma_q'=qI(\Gamma_q)$, as it is generated by the image of $z^q$.
Thus if  $p$ is an odd prime which divides $q$ then
$\Gamma_q'<X^p(\Gamma_q)$ and $c_{\Gamma_q}^p$ is injective.

More explicitly, 
let $f$ and $g$ be the homomorphisms from $\Gamma_q$ to $\mathbb{Z}$
defined by $f(x)=1$, $f(y)=f(z)=0$ and $g(y)=1$, $g(x)=g(z)=0$,
and let $\overline{f}_p, \overline{g}_p:\Gamma_q\to\mathbb{F}_p$ be their
mod-$(p)$ reductions.
Then the image of $f\cup{g}$ in $H^2(\Gamma_q;\mathbb{Q})$ is 0,
but $\overline{f}_p\cup\overline{g}_p\not=0$ \cite[Theorem 1]{Li91},
for each $p$ dividing $q$.
Thus $f\cup{g}$ generates the torsion subgroup of 
$H^2(\Gamma_q;\mathbb{Z})\cong\mathbb{Z}^2\oplus\mathbb{Z}/q\mathbb{Z}$.
In this case $c_{\Gamma_q}^\mathbb{Q}=0$,  so $\mu=0$,
but $c_{\Gamma_q}\not=0$,
$c_{\Gamma_q}^p\not=0$ and $\nu_p\not=0$.

Let $h:\Gamma_q\to\mathbb{F}_p$ be defined by $h(x)=h(y)=0$ and $h(z)=1$.
Reduction mod-$(p)$ maps $H^2(\Gamma_q;\mathbb{Z})$ onto 
$H^2(\Gamma_q;\mathbb{F}_p)$,  and the latter group has basis
\[
\{\overline{f}_p\cup{h},\overline{g}_p\cup{h},\overline{f}_p\cup\overline{g}_p\}.
\]
The first two elements are reductions of cohomology classes of infinite order.
However, $h$ does not lift to a homomorphism to $\mathbb{Z}$,
and these classes are not cup products of elements of $H^1(\Gamma_q;\mathbb{Z})$.
Thus we may have $\mathrm{Im}(c)\leq{tH}$, and so 
cannot treat the torsion and torsion-free parts separately.

\end{document}